\documentclass[reqno]{amsart}
\usepackage{amsmath}
\usepackage{amssymb}
\usepackage{amsthm}
\usepackage[arrow,matrix]{xy}
\begin{document}
\def\eq#1{{\rm(\ref{#1})}}
\theoremstyle{plain}
\newtheorem{thm}{Theorem}[section]
\newtheorem{lem}[thm]{Lemma}
\newtheorem{prop}[thm]{Proposition}
\newtheorem{cor}[thm]{Corollary}
\theoremstyle{definition}
\newtheorem{dfn}[thm]{Definition}
\newtheorem*{rem}{Remark}
\def\coker{\mathop{\rm coker}}
\def\ind{\mathop{\rm ind}}
\def\Re{\mathop{\rm Re}}
\def\vol{\mathop{\rm vol}}
\def\SO{\mathbin{\rm SO}}
\def\Im{\mathop{\rm Im}}
\def\min{\mathop{\rm min}}
\def\Spec{\mathop{\rm Spec}\nolimits}
\def\Hol{{\textstyle\mathop{\rm Hol}}}
\def\ge{\geqslant}
\def\le{\leqslant}
\def\C{{\mathbin{\mathbb C}}}
\def\R{{\mathbin{\mathbb R}}}
\def\N{{\mathbin{\mathbb N}}}
\def\Z{{\mathbin{\mathbb Z}}}
\def\D{{\mathbin{\mathcal D}}}
\def\H{{\mathbin{\mathcal H}}}
\def\M{{\mathbin{\mathcal M}}}
\def\al{\alpha}
\def\be{\beta}
\def\ga{\gamma}
\def\de{\delta}
\def\ep{\epsilon}
\def\io{\iota}
\def\ka{\kappa}
\def\la{\lambda}
\def\ze{\zeta}
\def\th{\theta}
\def\vp{\varphi}
\def\si{\sigma}
\def\up{\upsilon}
\def\om{\omega}
\def\De{\Delta}
\def\Ga{\Gamma}
\def\Th{\Theta}
\def\La{\Lambda}
\def\Om{\Omega}
\def\ts{\textstyle}
\def\sst{\scriptscriptstyle}
\def\sm{\setminus}
\def\na{\nabla}
\def\pd{\partial}
\def\op{\oplus}
\def\ot{\otimes}
\def\bigop{\bigoplus}
\def\iy{\infty}
\def\ra{\rightarrow}
\def\longra{\longrightarrow}
\def\dashra{\dashrightarrow}
\def\t{\times}
\def\w{\wedge}
\def\bigw{\bigwedge}
\def\d{{\rm d}}
\def\bs{\boldsymbol}
\def\ci{\circ}
\def\ti{\tilde}
\def\ov{\overline}
\def\md#1{\vert #1 \vert}
\def\nm#1{\Vert #1 \Vert}
\def\bmd#1{\big\vert #1 \big\vert}
\def\cnm#1#2{\Vert #1 \Vert_{C^{#2}}} 
\def\lnm#1#2{\Vert #1 \Vert_{L^{#2}}} 
\def\bnm#1{\bigl\Vert #1 \bigr\Vert}
\def\bcnm#1#2{\bigl\Vert #1 \bigr\Vert_{C^{#2}}} 
\def\blnm#1#2{\bigl\Vert #1 \bigr\Vert_{L^{#2}}} 
\title[Deformations of AC Special Lagrangian Submanifolds]{Deformations of
Asymptotically Cylindrical Special Lagrangian Submanifolds with
Fixed Boundary}

\author[Sema Salur and A. J. Todd]{Sema Salur and A. J. Todd}

\address {Department of Mathematics, University of Rochester, Rochester, NY, 14627}
\email{salur@math.rochester.edu }

\address {Department of Mathematics, University of Rochester,
Rochester, NY, 14627}
\email{ajtodd@math.rochester.edu }

\begin{abstract}
Given an asymptotically cylindrical special Lagrangian submanifold $L$ in an asymptotically cylindrical Calabi-Yau $3$-fold $X$, we determine conditions on a decay rate $\ga$ which make the moduli space of (local) special Lagrangian deformations of $L$ in $X$ a smooth manifold and show that it has dimension equal to the dimension of the image of $H^1_{cs}(L,\R)$ in $H^1(L,\R)$ under the natural inclusion map, $[\chi]\mapsto[\chi]$.
\end{abstract}

\date{}
\maketitle
\section{Introduction}
\label{sl1}

McLean \cite{McLe} proved that the moduli space of special Lagrangian deformations of a compact special Lagrangian submanifold $L$ of a Calabi-Yau $n$-fold is a smooth manifold of dimension $b^1(L)$, the first Betti number of $L$. The goal of this paper is to prove the following result for an asymptotically cylindrical special Lagrangian submanifold of an asymptotically cylindrical Calabi-Yau $3$-fold:

\begin{thm} Assume that $(X,\om,\Om,g)$ is an asymptotically cylindrical Calabi-Yau $3$-fold, asymptotic to $M\times S^1\times(R,\infty)$ with decay rate $\al<0$, where $M$ is a compact, connected $K3$ surface, and that $L$ is an asymptotically cylindrical special Lagrangian $3$-submanifold in $X$ asymptotic to $N\times\{p\}\times (R',\infty)$ for $R'>R$ with decay rate $\be$ for $\al\leq\be<0$ where $N$ is a compact special Lagrangian $2$-fold in $M$ and $p\in S^1$.

If $\ga<0$ is such that $\be<\ga$ and $(0,\ga^2]$ contains no eigenvalues of the Laplacian on functions on $N$ nor eigenvalues of the Laplacian on exact $1$-forms on $N$, then the moduli space $\M_L^\ga$ of asymptotically cylindrical special Lagrangian submanifolds in $X$ near $L$, and asymptotic to $N\times\{p\}\times(R',\infty)$ with decay rate $\ga$, is a smooth manifold of dimension $\dim V$ where $V$ is the image of the natural inclusion map of $H^1_{cs}(L,\R)\hookrightarrow H^1(L,\R)$.
\label{sl1thm}
\end{thm}

\begin{rem}
In \cite{JoSa}, a similar result is proven for asymptotically cylindrical coassociative submanifolds with fixed boundary in an asymptotically cylindrical $G_2$ manifold. The arguments presented in this paper parallel the arguments given therein.
\end{rem}

Examples of complete Calabi-Yau $n$-folds are constructed by Tian and Yau in \cite{TiYa1,TiYa2}. In general, they construct complete K\"ahler metrics with a prescribed Ricci curvature on quasi-projective manifolds $M$, that is, $M=\ov{M}\setminus D$ where $\ov{M}$ is a projective manifold and $D$ a smooth, ample divisor in $\ov{M}$. Then, under further assumptions that $D$ is anticanonical and $K^{-1}_{\ov{M}}$ is ample, $K_{\ov{M}}$ the canonical line bundle of $\ov{M}$, $M$ admits a complete Ricci-flat K\"ahler metric.

Examples of asymptotically cylindrical Calabi-Yau $3$-folds are then constructed by Kovalev \cite{Kova}. Therein Kovalev gives a construction of Riemannian metrics with holonomy $G_2$ on compact $7$-manifolds by first constructing asymptotically cylindrical Calabi-Yau $3$-folds. The construction of these asymptotically cylindrical Calabi-Yau $3$-folds essentially involves taking a compact, simply-connected K\"ahler manifold $\ov{W}$ and a compact complex surface $D$, where $D$, a $K3$ surface in $\ov{W}$, is an anticanonical divisor with trivial self-intersection class $D\cdot D=0$ in the second integral homology class $H_2(\ov{W},\Z)$ of $\ov{W}$, then writing the asymptotically cylindrical Calabi-Yau $3$-fold as $W=\ov{W}\setminus D$. Kovalev then obtains compact $G_2$-manifolds by crossing two such Calabi-Yau $3$-folds with two copies of $S^1$ to obtain two asymptotically cylindrical $G_2$-manifolds which are then glued together via connected sums.

Further, because of this construction, asymptotically cylindrical Calabi-Yau $3$-folds are of interest to physicists studying mirror symmetry. In particular, understanding the deformations of asymptotically cylindrical special Lagrangian submanifolds with a small decay rate inside these manifolds is a necessary part of being able to construct consistent topological quantum field theories.

We will rely heavily on the theory developed by Lockhart and McOwen \cite{Lock, LoMc} of weighted Sobolev spaces and elliptic operators on noncompact manifolds to determine the conditions on the decay rate $\ga$ and will be used frequently to prove smoothness of forms. It will also allow us to determine that certain asymptotically cylindrical partial differential operators are Fredholm, allowing us to ultimately use the Implicit Function Theorem for Banach Manifolds to prove our result.

The current paper is divided into three main sections: $1)$ Calabi-Yau and Special Lagrangian Geometry; $2)$ Analysis on Asymptotically Cylindrical Special Lagrangian $3$-Manifolds; $3)$ Proof of Theorem \ref{sl1thm}. Notably, the first section includes the defintions of cylindrical and asymptotically cylindrical Calabi-Yau $3$-folds, and cylindrical and asymptotically cylindrical special Lagrangian $3$-submanifolds, and the second section gives the material on weighted Sobolev spaces and elliptic operators on asymptotically cylindrical special Lagrangian $3$-manifolds.

\section{Calabi-Yau and Special Lagrangian Geometry}

We begin with the definitions of a Calabi-Yau $n$-fold and special Lagrangian $n$-submanifold, then present a sketch of the deformation theory of compact special Lagrangian submanifolds of Calabi-Yau $n$-folds; finally, we give the definitions of cylindrical and asymptotically cylindrical Calabi-Yau $3$-folds and cylindrical and asymptotically cylindrical special Lagrangian $3$-submanifolds. References for this section include: Harvey and Lawson, \cite{HaLa}; Joyce, \cite{Joyc1}; McLean, \cite{McLe}; and Kovalev, \cite{Kova}.

\subsection{Calabi-Yau $n$-folds and Special Lagrangian $n$-submanifolds}
\label{geo1}

\begin{dfn}
A \emph{complex} $n$-dimensional {\em Calabi-Yau} manifold $(X,\om,\Om,g)$ is a K\"ahler manifold with zero first Chern class, that is $c_1(X)=0$. In this case, $(X,\om,\Om,g)$ is also called a \emph{Calabi-Yau $n$-fold}.
\end{dfn}

\begin{rem}
The main point of the K\"ahler condition for our purposes is that $\om$ is closed. Also, as the complex structure plays virtually no role, we omit it from our notation.
\end{rem}

The condition that $c_1(X)=0$ is equivalent to the canonical bundle of $X$ being trivial which is true if and only if $X$ admits a nonvanishing holomorphic $(n,0)$-form $\Om$. This last condition actually reduces the holonomy group of $g$ to $SU(n)$ (or more generally to a subgroup of $SU(n)$) and is equivalent to having a nonvanishing \emph{complex} $(n,0)$-form $\Om$ such that $\na\Om=0$. We now define special Lagrangian submanifolds of a Calabi-Yau manifold using these structures:

\begin{dfn}
A \emph{real} $n$-dimensional submanifold $L\subseteq X$ is \emph{special Lagrangian} if $L$ is Lagrangian (i.e. $\om|_L\equiv 0$) and $\Im\Om$ restricted to $L$ is zero. In this case, we will also call $L$ a \emph{special Lagrangian $n$-submanifold}.
\end{dfn}

\begin{rem}
$\Re\Om$ restricts to be the volume form with respect to the induced metric on special Lagrangian submanifolds, so that the special Lagrangian submanifolds $L$ of a Calabi-Yau manifold $X$ are exactly the calibrated submanifolds of $X$ with respect to the calibration form $\Re\Om$ (see \cite[Section 3.1]{HaLa} for more information regarding this point).
\end{rem}

\subsection{Deformations of Compact Special Lagrangian $n$-folds}
\label{geo2}

We now sketch a proof of McLean's result on the moduli space of special Lagrangian deformations in the case of compact special Lagrangian submanifolds of a Calabi-Yau $n$-fold.

\begin{thm}
The moduli space of deformations of a smooth, compact, orientable special Lagrangian submanifold $L$ in a Calabi-Yau manifold $X$ within the class of special Lagrangian submanifolds is a smooth manifold of dimension equal to dim$(H^1(L))$.
\label{compact}
\end{thm}

\begin{proof}[Sketch of proof]
Define the deformation map $F:\nu_L\rightarrow\La^2T^*(L)\oplus\La^nT^*(L)$ from the space of smooth sections of the normal bundle on $L$ to the spaces of differential $2$-forms and differential $n$-forms on $L$ as follows:
$$F(V)=((\exp_V)^*(-\om), (\exp_V)^*(\Im\Om)).$$
Note that $F$ is just the restrictions of $-\om$ and $\Im\Om$ to $L_V$ which are then pulled back to $L$ via $(\exp_V)^*$ where $\exp_V$ is the exponential map giving a diffeomorphism of $L$ onto its image $L_V$ in a neighborhood of $0$.

Recall that the normal bundle $\nu_L$ of a special Lagrangian
submanifold is isomorphic to the cotangent bundle $T^*(L)$. Thus there is a natural identification of normal vector fields to $L$ with differential $1$-forms on $L$; moreover, since $L$ is compact, these normal vector fields can be identified with nearby submanifolds, so that, under these identifications, the kernel of $F$ corresponds to the special Lagrangian deformations.

The linearization of $F$ at $0$
$$\d F(0): \nu_L\rightarrow\Lambda^2T^*(L)\oplus\Lambda^nT^*(L)$$ is given by
\begin{equation*}
\begin{split}
\d F(0)(V)=
&\frac{\displaystyle\partial}{\displaystyle\partial{t}}F(tV)|_{t=0} =\frac{\displaystyle\partial}{\displaystyle\partial{t}}(\exp_{tV}^*(-\om), \exp_{tV}^*(\Im\Om))|_{t=0} \\
=&(-{\mathcal L}_{V}\om|_L, {\mathcal L}_{V}(\Im\Om)|_L).
\end{split}
\end{equation*}
Using the Cartan Formula for the Lie derivative ${\mathcal L}_{V}$, we find
\begin{equation*}
\begin{split}
\d F(0)(V)&=(-(i_V\d\om +\d(i_V\om))|_L,(i_V\d(\Im\Om)
+\d(i_V(\Im\Om)))|_L)\\
&=(-\d(i_V\om)|_L, \d(i_V(\Im\Om))|_L)=(\d v, \d*v),
\end{split}
\end{equation*}
where $i_V$ represents the interior derivative and $v$ is the dual $1$-form to the vector field $V$ with respect to the induced metric. Hence we have
\begin{equation*}
\d F(0)(V) =(\d v,\d*v)=(\d v,*\d^*v).
\end{equation*}

The next step is to show that $\d F(0)(V) =(\d v, \d*v)=(\d v, *\d^*v)$ is onto when considered as a map from $\nu_L$ to \emph{exact} $2$-forms and \emph{exact} $n$-forms. McLean shows this by first proving that $F$ is a map from $\nu_L$ to exact $2$-forms and exact $n$-forms as follows:

Note that the image of $F$ lies in closed $2$-forms and closed $n$-forms since $F$ is the pullback of the closed forms $\om$ and $\Im\Om$. Now, by replacing $V$ with $tV$, we see that $exp:L\rightarrow X$ is homotopic to the inclusion $i:L\rightarrow X$, so that $exp_V^*$ and $i^*$ induce the same map on the level of cohomology. Since $L$ is special Lagrangian, we get $[exp_V^*(\om)]=[i^*(\om)]=[\om |_L]=0$ and $[exp_V^*(\Im\Om)]=[i^*(\Im\Om)]=[\Im\Om|_L]=0$, so the forms in the image of $F$ are cohomologous to zero, that is, they are exact forms.

One can now easily show that, given any exact $2$-form $a$ and exact $n$-form $b$, there exists a $1$-form $v$ satisfying the
equations $\d v=a$ and $\d*v=b$; hence, $\d F(0)(V)$ is surjective as claimed. Finally, after completing the spaces of differential forms with appropriate norms, we can use the Implicit Function Theorem for Banach spaces and an elliptic regularity result to conclude that $F^{-1}(0,0)$ is a smooth manifold with tangent space at $0$ isomorphic to $H^1(L)$.
\end{proof}

\subsection{Asymptotically Cylindrical Calabi-Yau $3$-folds and Asymptotically Cylindrical Special Lagrangian $3$-submanifolds}
\label{geo3}

We now define cylindrical and asymptotically cylindrical Calabi-Yau manifolds. See \cite{Kova} for more information on the definitions in this section.

\begin{dfn}
A Calabi-Yau $3$-fold $(X_0,\om_0,\Om_0,g_0)$ is called {\em cylindrical} if $X_0=M\times S^1\times\R$ where $M$ is a connected, compact $K3$ surface with K\"ahler form $\ka_I$ and holomorphic $(2,0)$-form $\ka_J+i\ka_K$, and $(\om_0,\Om_0,g_0)$ is compatible with the product structure $M\times S^1\times\R$, that is, $\om_0=\ka_I+\d\th\w\d t$, $\Om_0=(\ka_J+i\ka_K)\w(\d\th+i\d t)$ and $g_0=g_{(M\times S^1)}+\d t^2$.
\label{acdfn1}
\end{dfn}

\begin{rem}
The indices on the $2$-forms $\ka_I$, $\ka_J$ and $\ka_K$ on $M$ are meant to reflect the hyperk\"ahler structure of $M$. (This is consistent with the notation in \cite{Kova}.)
\end{rem}

\begin{dfn}
A connected, complete Calabi-Yau $3$-fold $(X,\om,\Om,g)$ is called {\em asymptotically cylindrical with decay rate} $\al$ if there exists a cylindrical Calabi-Yau $3$-fold $(X_0,\om_0,\Om_0,g_0)$ as in Definition \ref{acdfn1}, a compact subset $K\subset X$, a real number $R>0$, and a diffeomorphism $\Psi: M\times S^1\times (R,\infty) \rightarrow X\setminus K$ such that $\Psi^*(\om)=\om_0 +\d\xi_1$ for some $1$-form $\xi_1$ with $\bmd{\na^k\xi_1}=O(e^{\al t})$ and $\Psi^*(\Om) =\Om_0 +\d\xi_2$ for some complex $2$-form $\xi_2$ with $\bmd{\na^k\xi_2}=O(e^{\al t})$ on $M\times S^1\times\R$ for all $k\ge 0$, where $\na$ is the Levi-Civit\`a connection of the cylindrical metric $g_0=g_{(M\times S^1)} + dt^2$.
\label{acdfn2}
\end{dfn}

Notice in this definition that we assume $X$ and $M$ are connected, so that $X$ only has one end; since we are working with Ricci-flat manifolds, this is not a restrictive assumption by a result obtained by the first author in \cite{Salu1}.

We now define cylindrical and asymptotically
cylindrical special Lagrangian submanifolds similarly.

\begin{dfn}
Let $(X_0,\om_0,\Om_0,g_0)$ be a cylindrical Calabi-Yau $3$-fold as in Definition \ref{acdfn1}. A $3$-dimensional submanifold $L_0$ of $X_0$ is called {\em cylindrical special Lagrangian} if $L_0=N\times \{p\}\times\R$ for some compact special Lagrangian submanifold $N$ in $M$, a point $p$ in $S^1$ and the restrictions of $\om_0$ and $\Im\Om_0$ to $L_0$ are zero, that is, $\om_0|_{L_0}=(\kappa_I+\d\th\w\d t)|_{L_0} =0$ and $\Im(\Om_0)|_{L_0}=\Im((\kappa_J+i\kappa_K)\w(\d\th+i\d t))|_{L_0}=0$.
\label{acdfn3}
\end{dfn}

\begin{dfn}
Let $(X_0,\om_0,\Om_0,g_0)$ , $M\times S^1\times\R$, $(X,\om,\Om,g)$, $K, R, \Psi$ and $\al$ be as in Definitions \ref{acdfn1} and \ref{acdfn2}, and let $L_0=N\times \{p\}\times\R$ be a cylindrical special Lagrangian $3$-fold in $X_0$ as in Definition \ref{acdfn3}.

A connected, complete special Lagrangian 3-fold $L$ in
$(X,\om,\Om,g)$ is called \emph{asymptotically cylindrical with decay rate} $\be$ for $\al\leq\be<0$ if there exists a compact
subset $K'\subset L$, a normal vector field $v$ on
$N\t\{p\}\t(R',\infty)$ for some $R'>R$, and a diffeomorphism
$\Phi:N\t\{p\}\t(R',\infty)\ra L\sm K'$ such that the following
diagram commutes:
\begin{equation}
\begin{gathered}
\xymatrix{M\t S^1\t (R',\iy) \ar[d]^\subset & N\t\{p\}\t (R',\iy)
\ar[l]^{\;\;\;\;\exp_v} \ar[r]_{\;\;\;\;\;\;\;\;\;\Phi} & (L\sm
K')
\ar[d]^\subset \\
M\t S^1\t (R,\iy)\ar[rr]^\Psi && (X\sm K), }
\end{gathered}
\label{aceq1}
\end{equation}
and $\bmd{\na^kv}=O(e^{\be t})$ on $N\t\{p\}\t(R',\iy)$ for
all $k\geq 0$.
\label{acdfn4}
\end{dfn}

Notice that here again, we require $L$ to be connected; however, unlike the case of the ambient manifold, it may be that $N$ is \emph{not} connected, so that $L$ may have multiple ends. This will not present a problem in the analysis since we are assuming the ends have the same decay rate.

\section{Analysis on Asymptotically Cylindrical Special Lagrangian $3$-Manifolds}

Lockhart and McOwen \cite{Lock, LoMc} setup a general framework for elliptic operators on noncompact manifolds where the basic tools are weighted Sobolev spaces. Using these spaces and the notion of asymptotically cylindrical linear elliptic partial differential operators, they get weighted Sobolev embedding theorems and recover elliptic regularity and Fredholm results. 
\begin{rem}
For an alternate approach to these types of results on noncompact manifolds, see Melrose \cite{Mel1, Mel2}. 
\end{rem}

We will begin with this theory in the specific case of asymptotically cylindrical special Lagrangian $3$-manifolds then study the asymptotically cylindrical linear elliptic partial differential operator $(\d+*\d^*)^p_{2+l,\ga}$ in this setting.

\subsection{Weighted Sobolev Spaces and Elliptic Operators}
\label{an1}

Let $L$ be an asymptotically cylindrical special Lagrangian $3$-manifold with data as in Definition \ref{acdfn4}.

\begin{dfn}
Let $A$ be a vector bundle on $L$ with smooth metrics $h$ on the fibers and a connection $\na_A$ on $A$ compatible with $h$; let $A_0$ be a vector bundle on $N\times\{p\}\t\R$ that is invariant under translations in $\R$, that is a \emph{cylindrical} vector bundle, with metrics $h_0$ on the fibers and $\na_{A_0}$ a connection on $A_0$ compatible with $h_0$ where $h_0$ and $\na_{A_0}$ are also invariant under translations in $\R$.

$A$, $h$ and $\na_A$ are said to be \emph{asymptotically cylindrical}, asymptotic to $A_0$, $h_0$ and $\na_{A_0}$ respectively, if $\Phi^*(A)\cong A_0$ on $N\t\{p\}\t(R',\infty)$ with $|\Phi^*(h)-h_0|=O(e^{\be t})$ and $|\Phi^*(\na_A)-\na_{A_0}|=O(e^{\be t})$ as $t\to\infty$.
\end{dfn}

Recall that for $k\geq 0$,  $$L^p_k(A)=\left\{s\in\Ga(A): s\text{ is }k\text{times weakly differentiable and }\Vert s\Vert_{L^p_k}<\infty\right\},$$
where $\Ga(A)$ is the space of \emph{all} sections of $A$,
$$\Vert s\Vert_{L^p_k}=\left(\sum_{j=0}^k \int_L \left|\na^j_As\right|^p \text{ } dvol_L\right)^{1/p}$$
and $dvol_L$ is the volume element of $L$, and $$L^p_{k,\rm{loc}}(A)=\{s\in\Ga(A):\phi s\in L^p_k(A) \text{ for all }\phi\in C^\infty_0(L)\},$$ where $C^\infty_0(L)$ is the set of all compactly-supported smooth functions on $L$.

\begin{dfn}
Choose a smooth function $\rho:L\to\R$ such that $\Phi^*(\rho)\equiv t$ on $N\times\{p\}\t(R',\infty)$. By Definition \ref{acdfn4}, this condition prescribes $\rho$ on $L\setminus K'$, so it is only necessary to smoothly extend $\rho$ over $K'$. For $p\geq 1$, $k\geq 0$ and $\ga\in\R$, the \emph{weighted Sobolev space} $L^p_{k,\ga}(A)$ is then the set of sections $s$ of $A$ such that $s\in L^p_{\rm{loc}}(A)$, $s$ is $k$ times weakly differentiable and $$\Vert s \Vert_{L^p_{k,\ga}}=\left(\sum_{j=0}^k \int_L e^{-\ga\rho}\left|\na^j_As\right|^p\text{ } dvol_L \right)^{1/p} < \infty.$$
In particular, $L^p_{k,\ga}(A)$ is a Banach space; further, note that different choices of $\rho$ give the same space $L^p_{k,\ga}(A)$ with equivalent norms since $\rho$ is uniquely determined outside of $K$.

Define the Banach space $C^k_{\ga}(A)$ of continuous sections $s$ of $A$ with $k$ continuous derivatives such that $e^{-\ga\rho}|\na^j_As|$ is bounded for each $j=0,1,\ldots,k$ where the norm is given by $$\Vert s \Vert_{C^k_{\ga}}=\sum_{j=0}^k \sup_{L}e^{-\ga\rho}\left|\na^j_As\right|.$$
\label{wss}
\end{dfn}

In general, from \cite[Theorem1.2]{Bart}, \cite[Theorem 3.10]{Lock} and \cite[Lemma 7.2]{LoMc} there is the following Weighted Sobolev Embedding Theorem (adapted to our case):

\begin{thm}[Weighted Sobolev Embedding Theorem]
Suppose that $k,l$ are integers with $k\geq l\geq 0$ and that $p,q,\ga,\ov{\ga}$ are real numbers with $p,q>1$. Then 
\begin{enumerate}
    \item If $\frac{k-l}{3}\geq\frac{1}{p}-\frac{1}{q}$ and $\ga\geq\ov{\ga}$ there is a continuous embedding of $L^p_{k,\ga}(A)\hookrightarrow L^q_{l,\ov{\ga}}(A)$;
    \item If $\frac{k-l}{3}>\frac{1}{p}$ and $\ga\geq\ov{\ga}$ there is a continuous embedding of $L^p_{k,\ga}(A)\hookrightarrow C^l_{\ov{\ga}}(A)$;
\end{enumerate}
\label{embedding}
\end{thm}

Now suppose that $A,B$ are two asymptotically cylindrical vector bundles on $L$ which are asymptotic to the cylindrical vector bundles $A_0,B_0$ on $N\times\{p\}\t\R$. Let $F_0:C^{\infty}(A_0)\to C^{\infty}(B_0)$ be a cylindrical linear partial differential operator of order $k$, that is, a linear partial differential operator which is invariant under translations in $\R$, from \emph{smooth} sections $C^{\infty}(A_0)$ of $A_0$ to \emph{smooth} sections $C^{\infty}(B_0)$ of $B_0$. Let $F:C^{\infty}(A)\to C^{\infty}(B)$ be a linear partial differential operator of order $k$ on $L$.

\begin{dfn}
$F$ is said to be an \emph{asymptotically cylindrical operator}, asymptotic to $F_0$, if, under the identifications $\Phi^*(A)\cong A_0$, $\Phi^*(B)\cong B_0$ on $N\times\{p\}\t\R$, $|\Phi^*(F)-F_0|=O(e^{\be t})$ as $t\to\infty$.
\end{dfn}

From these definitions $F$ extends to a bounded linear operator
$$F^p_{k+l,\ga}:L^p_{k+l,\ga}(A)\to L^p_{l,\ga}(B)$$
for all $p>1$, $l\geq 0$ and $\ga\in\R$. We then have the following elliptic regularity result \cite[Theorem 3.7.2]{Lock}:

\begin{thm}
Let $F$ and $F_0$ be as above. If $1<p<\infty$, $l\in\Z$ and $\ga\in\R$, then for all $s\in L^p_{k+l,loc}(A)$,
$$\Vert s\Vert_{L^p_{k+l,\ga}}\leq C(\Vert Fs\Vert_{L^p_{l,\ga}}+\Vert s\Vert_{L^p_{l,\ga}})$$
for some $C$ independent of $s$.
\label{reg}
\end{thm}

\begin{dfn}
Assume now that $F,F_0$ are also elliptic on $L$ and $N\times\{p\}\t\R$ respectively. Extend $F_0$ to $F_0:C^{\infty}(A_0\ot_{\R}\C)\to C^{\infty}(B_0\ot_{\R}\C)$, and define $\D_{F_0}$ as the set of $\ga\in\R$ such that for some $\de\in\R$ there exists a nonzero section $s\in C^{\infty}(A_0\ot_{\R}\C)$, invariant under translations in $\R$, such that $F_0(e^{(\ga+i\de)t}s)=0$.
\end{dfn}

\noindent The conditions for $F^p_{k+l,\ga}$ to be Fredholm are now given by \cite[Theorem 1.1]{LoMc}.

\begin{thm}
Using the setup of this section, $\D_{F_0}$ is a discrete subset of $\R$, and for $p>1$, $l\geq 0$ and $\ga\in\R$, $F^p_{k+l,\ga}:L^p_{k+l,\ga}(A)\to L^p_{l,\ga}(B)$ is Fredholm if and only if $\ga\not\in\D_{F_0}$.
\label{fred}
\end{thm}

This theorem and Theorem \ref{reg} imply that $\ker(F^p_{k+l,\ga})$ is a finite-dimensional vector space of smooth sections of $A$ whenever $\ga\not\in\D_{F_0}$, and from the Weighted Sobolev Embedding Theorem and the fact that $\ker(F^p_{k+l,\ga})$ is invariant under small changes in $\ga$ one can show:

\begin{lem}
If $\ga\not\in\D_{F_0}$, then the kernel $\ker(F^p_{k+l,\ga})$ is the same for any choices of $p>1$ and $l\geq 0$ and is a finite-dimensional vector space consisting of smooth sections of $A$.
\label{ker}
\end{lem}

Let $F^*:C^{\infty}(B)\to C^{\infty}(A)$ be the formal adjoint of $F$; that is, $F^*$ is the asymptotically cylindrical linear elliptic partial differential operator of order $k$ on $L$ such that
$$\langle Fs,\tilde{s}\rangle_{L^2(B)}=\langle s,F^*\tilde{s}\rangle_{L^2(A)}$$
for all compactly-supported $s\in C^{\infty}(A)$, $\tilde{s}\in C^{\infty}(B)$.

Then for $p>1$, $l\geq 0$ and $\ga\not\in\D_{F_0}$, $(F^*)^q_{-l,-\ga}:L^q_{-l,-\ga}(B)\to L^q_{-k-l,-\ga}(A)$ is the dual operator of $F^p_{k+l,\ga}$ where $q>1$ satisfies $\frac{1}{p}+\frac{1}{q}=1$, $L^q_{-l,-\ga}(B)$ and $L^q_{-k-l,-\ga}(A)$ are isomorphic as Banach spaces to the dual spaces $(L^p_{l,\ga}(B))^*$ and $(L^p_{k+l,\ga}(A))^*$ respectively and these isomorphisms identify $(F^*)^q_{-l,-\ga}$ with $(F^p_{k+l,\ga})^*$. Further, since Theorem \ref{reg} also holds for negative differentiability, $\ker((F^*)^q_{-l,-\ga})=\ker((F^*)^q_{k+m,-\ga})$ for all $m\in\Z$. This allows us to identify $\coker(F^p_{k+l,\ga})$ with $\ker((F^*)^q_{k+m,-\ga})^*$ for $\ga\not\in\D_{F_0}$, $p,q>1$ with $\frac{1}{p}+\frac{1}{q}=1$ and $l,m\geq 0$; moreover, the index of $F^p_{k+l,\ga}$ is then given by
\begin{equation}
\ind(F^p_{k+l,\ga})=\dim\ker(F^p_{k+l,\ga})-\dim\ker((F^*)^q_{k+m,-\ga}).
\label{ind1}
\end{equation}

\subsection{$\d+*\d^*$}
\label{an3}

Let $L$ be an asymptotically cylindrical special Lagrangian $3$-manifold with data as in Definition \ref{acdfn4}, and consider the asymptotically cylindrical linear elliptic operator
$$\d+*\d^*:C^{\infty}(T^*L)\to C^{\infty}(\La^2T^*L)\op C^{\infty}(\La^3T^*L),$$
with formal adjoint given by
$$\d^*+\d*:C^{\infty}(\La^2T^*L)\op C^{\infty}(\La^3T^*L)\to C^{\infty}(T^*L).$$
We study the extension
\begin{equation}
(\d+*\d^*)^p_{2+l,\ga}:L^p_{2+l,\ga}(T^*L)\to L^p_{1+l,\ga}(\La^2T^*L)\op L^p_{1+l,\ga}(\La^3T^*L)
\label{operator}
\end{equation}
for $p>1$, $l\geq 0$ and $\ga\in\R$. Suppose further that $L$ has no compact, connected components, so that $H^3(L,\R)=H^0_{cs}(L,\R)=0$; then $N$ is a compact, oriented $2$-manifold, and $L$ is the interior of a compact, oriented $3$-manifold $\ov{L}$ with boundary $\partial\ov{L}=N$.

From this we have the following long exact sequence in cohomology:
\begin{equation*}
\begin{split}
0\to H^0(L)\to H^0(N)\to H^1_{cs}(L)\to H^1(L)\to &H^1(N)\\
&\downarrow\\
0\leftarrow H^3_{cs}(L)\leftarrow H^2(N)\leftarrow H^2(L)\leftarrow &H^2_{cs}(L)
\end{split}
\end{equation*}
where $H^k(L)=H^k(L,\R)$ and $H^k(N)=H^k(N,\R)$ are the de Rham cohomology groups, $H^k_{cs}(L,\R)$ are compactly-supported de Rham cohomology groups and $b^k(L)$, $b^k(N)$ and $b^k_{cs}(L)$ the corresponding Betti numbers. If $V\subseteq H^1(L,\R)$ denotes the image of the natural map $H^1_{cs}(L,\R)\hookrightarrow H^1(L,\R)$, $[\chi]\mapsto[\chi]$, then, from the long exact sequence, $\dim(V)=b^1_{cs}(L)-b^0(N)+b^0(L)=b^2(L)-b^0(N)+b^0(L)$.

We now summarize a number of results from the previous subsection applied specifically to $(\d+*\d^*)^p_{2+l,\ga}$ as well as identify the kernel and cokernel of this operator for small $\ga<0$.

\begin{thm}
Suppose $\max\{\D_{(\d+*\d^*)_0}\cap(-\infty,0)\}<\ga<0$, $p,q>1$ with $\frac{1}{p}+\frac{1}{q}=1$ and $l,m\geq0$. Then the operator $(\d+*\d^*)^p_{2+l,\ga}$ is Fredholm with $\coker((\d+*\d^*)^p_{2+l,\ga})\cong(\ker((\d^*+\d*)^q_{2+m,-\ga}))^*$. The kernel $\ker((\d+*\d^*)^p_{2+l,\ga})$ is a vector space of smooth, harmonic $1$-forms, and the map $\ker((\d+*\d^*)^p_{2+l,\ga})\to H^1(L,\R)$, $\chi\mapsto[\chi]$, induces an isomorphism of $\ker((\d+*\d^*)^p_{2+l,\ga})$ with the image, $V$, of the natural inclusion map $H^1_{cs}(L,\R)\hookrightarrow H^1(L,\R)$; hence, $\dim\ker((\d+*\d^*)^p_{2+l,\ga})=\dim V$. Finally, the kernel $\ker((\d^*+\d*)^q_{2+m,-\ga})$ is a vector space of smooth coclosed $2$-forms and smooth harmonic $3$-forms.
\label{main}
\end{thm}

\begin{proof}
Since $\ga\not\in\D_{(\d+*\d^*)_0}$, $(\d+*\d^*)^p_{2+l,\ga}$ is Fredholm with $\coker((\d+*\d^*)^p_{2+l,\ga})\cong\ker((\d^*+\d*)^q_{2+m,-\ga})^*$ from Lemma \ref{fred} and the remarks following Lemma \ref{ker}.

For $\eta\in\ker((\d+*\d^*)^p_{2+l,\ga})$, $(0,0)\equiv(\d+*\d^*)\eta=(\d\eta,*\d^*\eta)$.
Since $0=*\d^*\eta$ if and only if $0=\d^*\eta$, $\eta$ is harmonic, and smoothness follows from Theorem \ref{reg}.

Let $\H^1$ denote the space of $1$-forms in $\ker(\d+\d^*)^p_{l+2,\ga}$ where $(\d+\d^*)^p_{2+l,\ga}:\op_{k=0}^n L^p_{2+l}(\La^kT^*L)\to\op_{k=0}^n L^p_{1+l}(\La^kT^*L)$. Then by \cite[Proposition 3.9]{JoSa}, the map $\H^1\to H^1(L,\R)$ is injective with image that of the natural inclusion map $H^1_{cs}(L,\R)\to H^1(L,\R)$, which is $V$ in the notation above. Since $\ker((\d+*\d^*)^p_{l+2,\ga})=\H^1$, we have the desired isomorphism.

Finally, let $(\eta_2,\eta_3)\in\ker((\d^*+\d*)^q_{2+m,-\ga})$, so that $\d^*\eta_2+\d*\eta_3=0$. Taking $\d^*$ of this equation yields $\d^*\d*\eta_3=0$. Further, since $*\eta_3$ is a function on $L$, $\d\d^**\eta_3=0$ which implies that $*\eta_3\in\ker((\d\d^*+\d^*\d)^q_{2+m,-\ga})=\ker((\d+\d^*)^q_{2+m,-\ga})$. Now, because $0=\d*\eta_3$ we see that $0=\d^*\eta_2+\d*\eta_3=\d^*\eta_2$, so $\eta_2$ is a coclosed $2$-form. Also, applying $-*$ to both sides of $0=\d*\eta_3$, we get $0=\d^*\eta_3$; thus, $\eta_3$ is a coclosed $3$-form. Last, we notice that $0=\d^*(*\eta_3)=*\d\eta_3$ from which it follows that $\eta_3$ is closed.  Smoothness, in both cases, follows from elliptic regularity, Theorem \ref{reg}.
\end{proof}

\begin{thm}
Let $p>1$, $l\geq 0$ and $\ga\in\R$. Then the operator (\ref{operator}) is \emph{not} Fredholm if and only if any of the following conditions hold:
\begin{enumerate}
    \item $\ga=0$,
    \item $\ga^2$ is a positive eigenvalue of $\De=\d^*\d$ on functions on $N$, or
    \item $\ga^2$ is a positive eigenvalue of $\De=\d^*\d+\d\d^*=\d\d^*$ on exact $1$-forms on $N$.
\end{enumerate}
\label{conditions}
\end{thm}

\begin{proof}
Throughout this proof let $\d^*_L,*_L$ operate on $L$ and $\d^*,*$ operate on $N$. A smooth section of $(T^*(N\t\{p\}\t\R)\ot_\R\C)$ which is invariant under translations in $\R$ can be written uniquely as $\eta+f\d t$ where $\eta\in C^{\infty}(T^*N\ot_{\R}\C)$, the smooth sections of $T^*N\ot_{\R}\C$, and $f:N\to\C$ is smooth. Then by Theorem \ref{fred}, $(\d+*_L\d^*_L)^p_{2+l,\ga}$ is not Fredholm if and only if there exist $f,\eta$ with $f,\eta$ not both zero and $\de\in\R$ such that $(\d+*_L\d^*_L)(e^{(\ga+i\de)t}(\eta+f\d t))\equiv(0,0)$. Note that $0\equiv*_L\d^*_L(e^{(\ga+i\de)t}(\eta+f\d t))$ if and only if $0\equiv\d^*_L(e^{(\ga+i\de)t}(\eta+f\d t))$. Thus:
$$0\equiv\d(e^{(\ga+i\de)t}(\eta+f\d t))=e^{(\ga+i\de)t}[-(\ga+i\de)\eta\w\d t+\d f\w\d t],$$
\noindent and
\begin{equation*}
\begin{split}
0&\equiv\d^*_L(e^{(\ga+i\de)t}(\eta+f\d t))=-*_L\d*_L(e^{(\ga+i\de)t}(\eta+f\d t))\\
&=-*_L\d[e^{(\ga+i\de)t}(*\eta\w\d t+fdvol_N)]\\
&=-*_L[(\ga+i\de)e^{(\ga+i\de)t}\d t\w fdvol_N+e^{(\ga+i\de)t}\d(*\eta\w\d t)]\\
&=-e^{(\ga+i\de)t}[(\ga+i\de)f+*\d*\eta]\\
&=-e^{(\ga+i\de)t}[(\ga+i\de)f-\d^*\eta].
\end{split}
\end{equation*}
This yields the following equations in $0$- and $1$-forms on $N$ respectively:
$$(\ga+i\de)f-\d^*\eta\equiv 0,$$
$$\d f-(\ga+i\de)\eta\equiv 0.$$

From here, we immediately see that if $\ga=0$ then $\de=0$, $\eta=0$ and $f\equiv 1$ yields a solution to the above system of equations, in which case $(\d+*_L\d^*_L)^p_{2+l,0}$ is not Fredholm. Assume now that $\ga+i\de\neq 0$. Since $N$ is compact, Hodge theory yields $\eta=(\eta_0, \eta_1,\eta_2)$ where $\eta_0$ is a harmonic $1$-form on $N$, $\eta_1$ is an exact $1$-form on $N$ and $\eta_2$ is a coexact $1$-form on $N$. Splitting the above system up by harmonic, exact and coexact $1$-forms on $N$ yields the following system of equations:
$$(\ga+i\de)\eta_0=0,$$
$$\d f-(\ga+i\de)\eta_1=0,$$
$$(\ga+i\de)\eta_2=0,$$
$$(\ga+i\de)f-\d^*\eta_1=0.$$

Because we are assuming that $\ga+i\de\neq 0$, $\eta_0,\eta_2=0$ which reduces the system to:
\begin{equation}
\d f=(\ga+i\de)\eta_1,
\label{system1}
\end{equation}
\begin{equation}
(\ga+i\de)f=\d^*\eta_1.
\label{system2}
\end{equation}
Notice that $f=0$ if and only if $\eta_1=0$. Since $f$ and $\eta$ cannot simultaneously be zero, neither can be zero.

Taking $\d^*$ of (\ref{system1}), then substituting (\ref{system2}) yields:
$$\d^*\d f=(\ga+i\de)\d^*\eta_1=(\ga+i\de)^2f,$$
so that $(\ga+i\de)^2$ is an eigenvalue of $\d^*\d$ on functions on $N$. Since such eigenvalues must be positive, this shows that $\de=0$, and so $\ga^2$ is an eigenvalue of $\d^*\d$ on functions on $N$. Conversely, assume $\De f=\d^*\d f=\ga^2 f$ for some nonzero $\ga\in\R$ and some smooth nonzero $f:N\to\C$; then $\eta=\ga^{-1}\d f$ is a smooth nonzero $1$-form on $N$ which satisfies the above equations.

Now, take $\d$ of (\ref{system2}) and substitute (\ref{system1}) to get:
$$\d\d^*\eta_1=(\ga+i\de)\d f=(\ga+i\de)^2\eta_1.$$
Because $\eta_1$ is an exact $1$-form on $N$, $(\d\d^*+\d^*\d)\eta_1=(\ga+i\de)^2\eta_1$ which shows that $(\ga+i\de)^2$ is an eigenvalue of $\d\d^*+\d^*\d=\d\d^*$ on exact $1$-forms on $N$, and so $\de=0$ by the same reasoning as above. Hence $\ga^2$ is an eigenvalue of $\d\d^*$ on exact $1$-forms on $N$. Conversely, assume that $\d\d^*\eta=\ga^2\eta$ for some nonzero $\ga\in\R$ and some smooth nonzero exact $1$-form $\eta$ on $N$; then $f=\ga^{-1}\d^*\eta$ is a smooth nonzero function on $N$ which satisfies (\ref{system1}) and (\ref{system2}).
\end{proof}

\section{Proof of Theorem \ref{sl1thm}}

Let $(X,g_X,\om,\Om)$ be an asymptotically cylindrical Calabi-Yau $3$-fold with decay rate $\al<0$, asymptotic to the cylindrical Calabi-Yau $3$-fold $(X_0,g_{X_0},\om_0,\Om_0)$, $X=M\times S^1\times \R$ where $M$ is a compact, connected $K3$ surface as in Definition \ref{acdfn1}. Let $K\subset X$ be a compact subset, $R>0$ and $\Psi:M\times S^1\times (R,\infty)\to X\setminus K$ a diffeomorphism with the following properties:
\begin{enumerate}
    \item $\Psi^*(\om_{X_0})=\om+\d\ze_1$, for some $\ze_1$, a $1$-form on $X$ with $|\na^k\ze_1|=O(e^{\al t})$ for all $k\geq 0$;
    \item $\Psi^*(\Om_{X_0})=\Om+\d\ze_2$, for some $\ze_2$, a complex $2$-form on $X$ with $|\na^k\ze_2|=O(e^{\al t})$ for all $k\geq 0$.
\end{enumerate}

Let $L$ be an asymptotically cylindrical special Lagrangian $3$-submanifold of $X$ with decay rate $\be$ ($\al\leq\be<0$), asymptotic to the cylindrical special Lagrangian $3$-submanifold $L_0=N\t\{p\}\t\R$ of $X_0$, where $N$ is a compact special Lagrangian $2$-submanifold of $M$ and $p\in S^1$ as in Definition \ref{acdfn4}. Let $K'\subset L$ be a compact subset, $R'>R$, $v$ a normal vector field on $N\t\{p\}\t (R',\infty)$ with $|\na^kv|=O(e^{\be t})$ for all $k\geq 0$ and $\Phi:N\t\{p\}\t (R',\infty)\to L\setminus K'$ a diffeomorphism making Diagram (\ref{aceq1}) commute.

Let $\ga<0$ be strictly less than $\be$ and be such that $(0,\ga^2]$ contains neither eigenvalues of the Laplacian $\De_N=\d^*\d$ on complex-valued functions on $N$ nor eigenvalues of the Laplacian $\De_N=\d\d^*$ on exact $1$-forms of $N$. Let $p>3$, $l\geq 1$ and the map $$(\d+*\d^*)^p_{2+l,\ga}: L^p_{2+l,\ga}(T^*L)\to L^p_{1+l,\ga}(\La^2T^*L)\op L^p_{1+l,\ga}(\La^3T^*L)$$ be as in Equation (\ref{operator}). By Theorem \ref{conditions}, the conditions on $\ga$ imply this operator is Fredholm, so the results of Theorem \ref{main} are applicable.

We begin by constructing an identification of small sections of the normal bundle of $L$ with $X$ near $L$ that is compatible with the data on these manifolds. Let $\nu_N$ be the normal bundle of $N$ in $M$ with exponential map $\exp_N:\nu_N\to M$; let $\ep>0$ such that $\exp_N:B_{2\ep}(\nu_N)\to T_N$ is a diffeomorphism of the subbundle $B_{2\ep}(\nu_N)$ whose fiber above each point is the ball of radius $2\ep$ about $0$, with a tubular neighborhood $T_N$ of $N$ in $M$. Then $B_{2\ep}(\nu_N)\t\{ p\}\t\R$ is a subbundle of the normal bundle $\nu_N\t\{p\}\t\R$, $T_N\t\{p\}\t\R$ is a tubular neighborhood (both of $N\t\{p\}\t\R$ in $M\t S^1\t\R$) and $\exp_N\t\io\t id:B_{2\ep}(\nu_N)\t\{p\}\t\R\to T_N\t\{p\}\t\R$ is a diffeomorphism, where $\io$ is the natural inclusion map and $id$ the identity map on $\R$. Notice that $v$ from above is a section of $\nu_N\t\{p\}\t (R',\infty)$, so since $v$ is decaying, we can assume that the graph of $v$ lies in $B_{2\ep}(\nu_N)\t\{p\}\t (R',\infty)$ (making $K',R'$ larger if necessary).

Let $\pi:B_{\ep}(\nu_N)\t\{p\}\t(R',\infty)\to N\t\{p\}\t(R',\infty)$ be the natural projection map, and define the map $\Xi:B_{\ep}(\nu_N)\t\{p\}\t (R',\infty)\to X\setminus K$ by $\Xi(w)=\Psi[(\exp_N\t\io\t id)(v|_{\pi(w)}+w)]$. First notice that since $\Xi$ is a composition of diffeomorphisms (onto their images), $\Xi$ is also a diffeomorphism (onto its image). Second, thinking of $N\t\{p\}\t(R',\infty)$ as the zero section of $B_{\ep}(\nu_N)\t\{p\}\t (R',\infty)$, $\Xi|_{N\t\{p\}\t (R',\infty)}=\Psi\circ(\exp_N\t\io\t id)|_{N\t\{p\}\t(R',\infty)}=\Psi\circ\exp_v=\Phi$ where the last two equalities follow from Definition \ref{acdfn4} and the commutativity of Diagram (\ref{aceq1}). Third, $\d\Xi:T(B_{\ep}(\nu_N)\t\{p\}\t (R',\infty))\to \Xi^*(T(X\setminus K))=\Phi^*(T(X\setminus K))$ is an isomorphism.

Since $\Phi:N\t\{p\}\t(R',\infty)\to L\setminus K'$ is a diffeomorphism, $\d\Phi:T(N\t\{p\}\t(R',)))\to T(L\setminus K')$ is an isomorphism. In fact, by the above discussion, $\d\Phi=\d\Xi|_{N\t\{p\}\t(R',\infty)}:T(B_{\ep}(\nu_N)\t \{p\}\t(R',\infty))|_{N\t\{p\}\t(R',\infty)}\to \Phi^*(T(L\setminus K'))$. Thus, define $\xi=\d\Xi|_{N\t\{p\}\t(R',\infty)}$, so that
\begin{equation}
\begin{split}
\xi&:\nu_N\t\{p\}\t(R',\infty)\cong \frac{T(B_{\ep}(\nu_N)\t\{p\}\t (R',\infty))|_{N\t\{p\}\t(R',\infty)}}{T(N\t\{p\}\t(R',\infty) )}\\
&\to\Phi^*(T(X\setminus K))/\Phi^*(T(L\setminus K'))\cong \Phi^*(T(X\setminus K)/T(L\setminus K'))=\Phi^*(\nu_L),\\
\label{xi}
\end{split}
\end{equation}
where $\nu_L$ is the normal bundle of $L$; $\xi$ is an isomorphism of the vector bundles $\nu_N\t\{p\}\t(R',\infty)$ and $\Phi^*(\nu_L)$ by construction.

Finally, let $\Th:B_{\ep'}(\nu_L)\to T_L$ denote an identification of $B_{\ep'}(\nu_L)$ (for some small $\ep'>0$) with $T_L$, a tubular neighborhood of $L$ in $X$, satisfying the following properties: first, thinking of $L$ as the zero section of $B_{\ep'}(\nu_L)$, $\Th|_L=id_L$; also, $\ep'$ should be small enough so that $\xi^*(\Phi^*(B_{\ep'}(\nu_L)))\subset B_{\ep}(\nu_N)\t\{p\}\t(R',\infty)\subset\nu_N\t\{p\}\t(R',\infty)$ and $\Th\circ\xi=\Xi$ on $\xi^*(\Phi^*(B_{\ep'}(\nu_L)))$. Notice that the first condition implies that $\d\Th|_L=id_{TL}:TL\subset T(B_{\ep'}(\nu_L))|_L\to TL\subset T(T_L)|_L$; the last condition uniquely defines $\Th$ and $T_L$ on $B_{\ep'}(\nu_L)|_{L\setminus K'}$, so one need only smoothly extend $\Th$ and $T_L$ to the compact subset $K'$ of $L$.

To summarize this construction, we have used the maps $\Xi$ and $\xi$ to define $\Th$, giving an identification of small sections of the normal bundle $\nu_L$ of $L$ with the ambient manifold $X$ near $L$, in such a way that is compatible with the diffeomorphisms $\Phi$ and $\Psi$. This allows us to identify small sections of $\nu_L$ with $3$-submanifolds of $X$ near $L$ and detect the asymptotic convergence of such a submanifold to $N\t\{p\}\t(R',\infty)$ by the asymptotic convergence of small sections of $\nu_L$ to zero. Further, since $\nu_L\cong TL$ on the special Lagrangian submanifold $L$, we will regard $\Th:B_{\ep'}(T^*L)\to T_L$.

This has the advantage that now smooth sections $\eta$ of the space $L^p_{2+l,\ga}(B_{\ep'}(T^*L))$ (note that this is an open Banach subspace of the Banach manifold $L^p_{2+l,\ga}(T^*L)$, so it is itself a Banach manifold) now correspond to smooth $3$-submanifolds of $X$ near $L$, and since $\eta$ is by definition a map $\eta:L\to L^p_{2+l,\ga}(B_{\ep'}(T^*L))$, $\Th\circ\eta:L\to T_L$. Hence we define our deformation map as $F:L^p_{2+l,\ga}(B_{\ep'}(T^*L))\to \La^2T^*L\op\La^3T^*L$, $F(\eta)=[(\Th\circ\eta)^*(-\om),(\Th\circ\eta)^*(\Im\Om)]$. Letting $\Ga_{\eta}$ denote the graph of $\eta$ in $L^p_{2+l,\ga}(B_{\ep'}(T^*L))$ and $\tilde{L}=\Th(\Ga_{\eta})$ its image in $X$, $\tilde{L}$ is special Lagrangian precisely when $\om|_{\tilde{L}}\equiv0$ and $\Im\Om|_{\tilde{L}}\equiv0$, but this is equivalent to $F(\eta)=(0,0)$. Thus, $F^{-1}(0,0)$ parameterizes the special Lagrangian $3$-submanifolds $\tilde{L}$ near $L$.

This completes our setup. Our first step now will be to prove that $F$ extends to the smooth map of Banach manifolds $F=F^p_{2+l,\ga}:L^p_{2+l,\ga}(B_{\ep'}(T^*L))\to L^p_{1+l,\ga}(\La^2T^*L)\op L^p_{1+l,\ga}(\La^3T^*L)$. We study this extension because its linearization at $0$ is given by the operator $\d F^p_{2+l,\ga}(0)=(\d+*\d^*)^p_{2+l,\ga}$ from above; moreover, we will show that $F$ actually maps into the image of $(\d+*\d^*)^p_{2+l,\ga}$. The point is to ultimately use these results to invoke the Implicit Mapping Theorem for Banach Manifolds (see, e. g., \cite[Theorem 1.2.5]{Joyc1}).

\begin{prop}
$F:L^p_{2+l,\ga}(B_{\ep'}(T^*L))\to L^p_{1+l,\ga}(\La^2T^*L)\op L^p_{1+l,\ga}(\La^3T^*L)$ is a smooth map of Banach manifolds with linearization at $0$ given by $\eta\mapsto(\d\eta,*\d^*\eta)$.
\label{Fmap}
\end{prop}

\begin{proof}
Begin by noting that our assumptions $p>3$ and $l\geq 1$ yield, by the Weighted Sobolev Embedding Theorem \ref{embedding}, the continuous inclusion $L^p_{2+l,\ga}\hookrightarrow C^1_{\ga}$, so that locally $F(\eta)\in L^p_{1+l}$. Now $F(\eta)$ is simply the restriction of $\Th^*(-\om)$, a $2$-form on $B_{\ep'}(T^*L)$, and of $\Th^*(\Im\Om)$, a $3$-form on $B_{\ep'}(T^*L)$, to $\Ga_{\eta}$. From the properties of $\Th$, this means that $F(\eta)$ is equal to $\Xi^*(-\om)$ and $\Xi^*(\Im\Om)$ on $B_{\ep'}(\nu_N)\t\{p\}\t(R',\infty)$; further, the asymptotic properties of $\Psi$, $\Phi$ and $v$  we built into $\Xi$ imply that $\Xi^*(-\om)$ is the sum of a translation-invariant $2$-form on $B_{\ep'}(\nu_N)\t\{p\}\t(R',\infty)$, the pullback of the negative of the cylindrical K\"ahler $2$-form $\om_0$ on $M\t S^1\t\R$ to $B_{\ep'}(\nu_N)\t\{p\}\t(R',\infty)$ and an error term which decays at a rate of $O(e^{\be t})$; similarly, $\Xi^*(\Im\Om)$ is the sum of a translation-invariant $3$-form on $B_{\ep'}(\nu_N)\t\{p\}\t(R',\infty)$, the pullback of the $3$-form $\Im\Om$ on $M\t S^1\t\R$ to $B_{\ep'}(\nu_N)\t\{p\}\t(R',\infty)$ and an error term which decays at a rate of $O(e^{\be t})$.

Hence because $\be<\ga$, $F(\eta)\in L^p_{1+l,\ga}(\La^2T^*L)\op L^p_{1+l,\ga}(\La^3T^*L)$. Smoothness of $F$ follows by construction, so the first claim follows. That the linearization of $F$ at $0$ is $(\d+*\d^*)^p_{2+l,\ga}$ follows exactly as in Theorem \ref{compact} since the calculation is local.
\end{proof}

\begin{lem}
The image of $F$ lies in exact $2$-forms and exact $3$-forms; specifically,
\begin{equation*}
\begin{split}
F(L^p_{2+l,\ga}(B_{\ep'}(T^*L)))&\subset \d(L^p_{1+l,\ga}(T^*L))\op\d(L^p_{1+l,\ga}(\La^2T^*L))\\
&\subset L^p_{l,\ga}(\La^2T^*L)\op L^p_{l,\ga}(\La^3T^*L).
\end{split}
\end{equation*}
\end{lem}

\begin{proof}
Recall that $\om$ and $\Im\Om$ are closed forms, so they determine the de Rham cohomology classes $[\om]$ and $[\Im\Om]$. In particular, $\om|_L\equiv 0$ and $\Im\Om|_L\equiv 0$ since $L$ is special Lagrangian, so $[\om|_L]=0$ and $[\Im\Om|_L]=0$; moreover, since $T_L$, the tubular neighborhood of $L$ from above, retracts onto $L$, $[\om|_{T_L}]=[\om|_L]$ and $[\Im\Om|_{T_L}]=[\Im\Om|_L]$. Thus, there exists $\tau_1\in C^{\infty}(T^*T_L)$ such that $\om|_{T_L}=\d\tau_1$ and $\tau_2\in C^{\infty}(\La^2T^*T_L)$ such that $\Im\Om|_{T_L}=\d\tau_2$. Now, since $\om|_{L}\equiv 0$, we can assume that $\tau_1|_L\equiv 0$; second, because $\om$ and all its derivatives decay at a rate $O(e^{\be t})$ to the translation-invariant $2$-form $\om_0$ on $M\t S^1\t\R$, we can assume that $\tau_1$ and all its derivatives decay at a rate $O(e^{\be t})$ to a translation invariant $1$-form on $T_N\t\{p\}\t\R$. We can make similar assumptions regarding $\tau_2$ based on the properties of $\Im\Om$.

From this, we calculate the following for $\eta\in L^p_{2+l,\ga}(T^*L)$:
\begin{equation*}
\begin{split}
F(\eta)&=((\Th\circ\eta)^*(-\om),(\Th\circ\eta)^*(\Im\Om))\\
&=((\Th\circ\eta)^*(-\d\tau_1),(\Th\circ\eta)^*(\d\tau_2))\\
&=(\d(\Th\circ\eta)^*(-\tau_1)),\d((\Th\circ\eta)^*(\tau_2))).
\end{split}
\end{equation*}
The result now follows by Proposition \ref{Fmap}.
\end{proof}

\begin{prop}
Let $\mathcal{C}$ denote the image of the operator $(\d+*\d^*)^p_{2+l,\ga}$. Then $$F:L^p_{2+l,\ga}(B_{\ep'}(T^*L))\to\mathcal{C}.$$
\label{image}
\end{prop}

\begin{proof}
By Theorem \ref{main}, $\coker((\d+*\d^*)^p_{2+l,\ga})\cong(\ker((\d^*+\d*)^q_{2+m,-\ga}))^*$ with $\frac{1}{p}+\frac{1}{q}=1$ and $m\geq 1$, so $F(\eta)\in\mathcal{C}$ if and only if $$\langle F(\eta),(\chi_1,\chi_2)\rangle_{L^2}\equiv0\text{ for all }(\chi_1,\chi_2)\in\ker((\d^*+\d*)^q_{2+m,-\ga}).$$ By the previous lemma $F(\eta)=(\d\tau_1,\d\tau_2)$ for some $\tau_1\in L^p_{1+l,\ga}(T^*L)$ and $\tau_2\in L^p_{1+l,\ga}(\La^2T^*L)$; then Theorem \ref{main} implies $\chi_1$ and $\chi_2$ are coclosed $2$- and $3$-forms respectively which yields the following:
\begin{equation*}
\begin{split}
\langle F(\eta),(\chi_1,\chi_2)\rangle_{L^2}&=\langle \d\tau_1,\chi_1\rangle_{L^2}+\langle\d\tau_2,\chi_2\rangle_{L^2}\\
&=\langle \tau_1,\d^*\chi_1\rangle_{L^2}+\langle\tau_2,\d^*\chi_2\rangle_{L^2}=0.
\end{split}
\end{equation*}
\end{proof}

The next step is to use the Implicit Mapping Theorem for Banach Spaces. Let $\mathcal{A}=\ker((\d+*\d^*)^p_{2+l,\ga})$ and $\mathcal{B}$ denote the subspace of $L^p_{2+l,\ga}(T^*L)$ that is $L^2$-orthogonal to $\mathcal{A}$. Because $\mathcal{A}$ is finite-dimensional and the $L^2$-inner product is continuous on $L^p_{2+l,\ga}(T^*L)$, $\mathcal{A}$ and $\mathcal{B}$ are Banach spaces such that $\mathcal{A}\op\mathcal{B}=L^p_{2+l,\ga}(T^*L)$. Choose open neighborhoods $\mathcal{U}$, $\mathcal{V}$ of $0$ in $\mathcal{A}$, $\mathcal{B}$, respectively, such that $\mathcal{U}\t\mathcal{V}\subset L^p_{2+l,\ga}(B_{\ep'}(T^*L))$. Then, by Proposition \ref{Fmap} and Proposition \ref{image}, $F:\mathcal{U}\t\mathcal{V}\to\mathcal{C}$ is a smooth map of Banach manifolds, $F(0,0)=(0,0)$ and $\d F(0,0)=(\d+*\d^*)^p_{2+l,\ga}:\mathcal{A}\op\mathcal{B}\to\mathcal{C}$; moreover, $(\d+*\d^*)^p_{2+l,\ga}|_{\mathcal{B}}:\mathcal{B}\to\mathcal{C}$ is an isomorphism of vector spaces by construction, and it is a homeomorphism of topological spaces by the Open Mapping Theorem. The Implicit Mapping Theorem for Banach Spaces now guarantees the existence of a connected open neighborhood $\mathcal{U'}\subset\mathcal{U}$ of $0$ and a smooth function $G:\mathcal{U'}\to\mathcal{V}$ such that $G(0)=0$ and $F(x)=(x,G(x))\equiv(0,0)$ for all $x\in\mathcal{U'}$. Hence we conclude that near $(0,0)$, $F^{-1}(0,0)=\{(x,G(x)):x\in\mathcal{U'}\}$, so that $F^{-1}(0,0)$ is smooth, finite-dimensional and locally isomorphic to $\mathcal{A}=\ker((\d+*\d^*)^p_{2+l,\ga})$.

The last part of the proof consists of defining a map from $F^{-1}(0,0)$ to the moduli space $M^{\ga}_L$ of asymptotically cylindrical special Lagrangian deformations of $L$ near $L$; however, one technical step involved in showing the map is well defined is to show that the sections $\eta$ in $F^{-1}(0,0)$ are smooth. Theorem \ref{sl1thm} will then follow from these results, Theorem \ref{main} and the fact that $F^{-1}(0,0)$ is smooth, finite-dimensional and locally isomorphic to $\ker((\d+*\d^*)^p_{2+l,\ga})$.

\begin{lem}
If $\eta\in F^{-1}(0,0)$, then $\eta\in L^p_{2+m,\ga}(T^*L)$ for all $m\geq 1$.
\label{smooth}
\end{lem}

\begin{proof}
We begin by noting that the functional form of $F(\eta)$ is given by $H(x,\eta|_x,\na\eta|_x)$ where $x\in L$ and $H$ is a smooth function. Fix $m\geq 1$, and let $\na$ denote the Levi-Civit\`a connection of $g_L$ on $L$. We are going to apply the Laplacian $\De_L=g_L^{ij}\na_i\na_j$ to $F$ which will allow us to split $F$ up in such a way that we can use a regularity result (which we will prove in the course of this argument) to increase the regularity of $\eta$.

Let $\na^x$ denote the derivative in the $x$-direction; let $\pd^y$ and $\pd^z$ denote the derivatives in the $y$- and $z$-directions respectively, where $y=\eta$ and $z=\na\eta$. Then

$$\De_L(F(\eta))=g^{ij}_L\na_i\na_j(H(x,y,z))$$
$$=g^{ij}_L\na_i[(\na^x_jH)(x,y,z)+(\pd^yH)(x,y,z)\cdot \na_j\eta+(\pd^zH)(x,y,z)\cdot\na_j\na\eta]$$
$$=g^{ij}_L[(\na^x_i\na^x_jH)(x,y,z)+(\na^x_i\pd^yH)(x,y,z)\cdot\na_j\eta +(\pd^yH)(x,y,z)\cdot\na_i\na_j\eta$$
$$+(\na^x_i\pd^zH)(x,y,z)\cdot\na_j\na\eta+(\pd^zH)(x,y,z)\cdot\na_i\na_j\na\eta$$ $$+(\pd^y\na^x_jH)(x,y,z)\cdot\na_i\eta+(\pd^y\pd^yH)(x,y,z)\cdot (\na_i\eta\ot\na_j\eta)$$
$$+(\pd^y\pd^zH)(x,y,z)\cdot(\na_i\eta\ot\na_j\na\eta) +(\pd^z\na^x_jH)(x,y,z)\cdot\na_i\na\eta$$
$$+(\pd^z\pd^yH)(x,y,z)\cdot (\na_i\na\eta\ot\na_j\eta)+(\pd^z\pd^zH)(x,y,z)\cdot(\na_i\na\eta\ot\na_j \na\eta)]$$
$$=(\pd^zH)(x,y,z)\cdot\De_L\na\eta+(\pd^yH)(x,y,z)\cdot\De_L\eta +g^{ij}_L[(\na^x_i\na^x_jH)(x,y,z)$$
$$+(\na^x_i\pd^yH)(x,y,z)\cdot\na_j\eta +(\na^x_i\pd^zH)(x,y,z)\cdot\na_j\na\eta$$
$$+(\pd^y\na^x_jH)(x,y,z)\cdot\na_i\eta+(\pd^y\pd^yH)(x,y,z)\cdot (\na_i\eta\ot\na_j\eta)$$
$$+(\pd^y\pd^zH)(x,y,z)\cdot(\na_i\eta\ot\na_j\na\eta) +(\pd^z\na^x_jH)(x,y,z)\cdot\na_i\na\eta$$
$$+(\pd^z\pd^yH)(x,y,z)\cdot (\na_i\na\eta\ot\na_j\eta)+(\pd^z\pd^zH)(x,y,z)\cdot(\na_i\na\eta\ot\na_j \na\eta)].$$
\newline
Notice that $\De_L$ splits $F$ into two pieces: the only term $(\pd^zH)(x,y,z)\cdot\De_L\na\eta$ involving the third derivatives of $\eta$, and everything else which depends only on $\eta$ up to its second derivatives, which we will denote from now on by $E(x,\eta,\na\eta,\na^2\eta)$. Now, for $\eta$ fixed, let $\tilde{H}_{\eta}:L^p_{3+s}(T^*L)\to L^p_s(\La^2T^*L)\op L^p_s(\La^3T^*L)$ denote the map defined by $\si\mapsto\tilde{H}_{\eta}(\si)=(\pd^zH)(x,\eta|_x,\na\eta|_x) \cdot\De_L\na\si|_x$, so that $\tilde{H}_{\eta}$ is a third-order linear elliptic operator ($\tilde{H}_{\eta}$ is basically $\De_L\d$ with $\d$ the exterior derivative operator). Further, since the coefficients of $\tilde{H}_{\eta}$ only depend on $\eta$ and $\na\eta$, they are $L^p_{1+m}$ locally, and so the maximum regularity we can get from $\tilde{H}_{\eta}(\si)$ is $L^p_{1+m}$; this forces us to take $s$ such that $0\leq s\leq m+1$. Of course, since the coefficients of $\tilde{H}_{\eta}$ are only $L^p_{1+m}$ locally (rather than smooth) we cannot use the elliptic regularity result, Theorem \ref{reg}, so we need the following result:

\begin{lem}
Assume $\si\in L^p_{3}(T^*L)$ and $\tilde{H}_{\eta}(\si)\in L^p_{m}(\La^2T^*L)\op L^p_{m}(\La^3T^*L)$. Then $\si\in L^p_{3+m}(T^*L)$, and there exists a constant $\tilde{C}>0$ such that $$\Vert \si\Vert_{L^p_{3+m,\ga}}\leq \tilde{C}(\Vert \tilde{H}_{\eta}(\si)\Vert_{L^p_{m,\ga}}+\Vert \si\Vert_{L^p_{3,\ga}}).$$
\end{lem}

\begin{proof}
We will use results from Morrey \cite[Section 6.2]{Morr} to prove this lemma. If the coefficients of $\tilde{H}_{\eta}$ are $C^m$, then \cite[Theorem 6.2.5]{Morr} guarantees that $\si$ is locally $L^p_{3+m}$. Further, \cite[Theorem 6.2.6]{Morr} provides a local interior estimate of the form above where $\tilde{C}>0$ depends on $m,p$, the domains involved, $C^m$-bounds on the coefficients of $\tilde{H}_{\eta}$ and a modulus of continuity for their $m$-th derivatives. Notice that if we have H\"older $C^{0,\al}$, $\al\in(0,1)$, bounds for the $m$-th derivatives, we get our modulus of continuity; we can further simplify the problem to finding $C^{1+m,\al}$ bounds for $\eta$ since this will give $C^{m,\al}$ bounds for $\tilde{H}_{\eta}$ giving us the desired bounds. Recall that we are assuming $p>3$, so that with $\al=1-p/3$, the Sobolev Embedding Theorem embeds $L^p_{2+m}$ into $C^{1+m,\al}$. Since $\eta\in L^p_{2+m,\ga}$ and $\ga<0$, we get the desired control on $\eta$ which yields the modulus of continuity and the result.
\end{proof}

To finish the proof, let $\eta\in F^{-1}(0,0)$, so that $\eta\in L^p_{2+m,\ga}(T^*L)$. From the above computation and the fact that $F(\eta)=0$, we have $\tilde{H}_{\eta}(\eta)=-E(x,\eta,\na\eta,\na^2\eta)$, so that $\tilde{H}_{\eta}(\eta)\in L^p_{m,\ga}(\La^2T^*L)\op L^p_{m,\ga}(\La^3T^*L)$. By the regularity result above, we now have that $\eta\in L^p_{3+m,\ga}(T^*L)$. The result now follows simply from induction on $m$.
\end{proof}

\begin{prop}
Let $M^{\ga}_L$ denote the moduli space of nearby asymptotically cylindrical special Lagrangian deformations of $L$ with decay rate $\ga$ and asymptotic to $N\t\{p\}\t(R',\infty)$. Define $S:F^{-1}(0,0)\to\{3\text{-submanifolds of }X\}$ by $\eta\mapsto\Th(\Ga_{\eta})$ where $\Ga_{\eta}$ is the graph of $\eta$ in $B_{\ep'}(T^*L)$. Then $S$ is a homeomorphism of $F^{-1}(0,0)$ with a neighborhood of $L$ in $M^{\ga}_L$.
\end{prop}

\begin{proof}
Let $\eta\in F^{-1}(0,0)$, let $\tilde{L}=S(\eta)\subset T_L\subset X$. By the Lemma \ref{smooth}, $\tilde{L}$ is smooth; further, $$(0,0)=F(\eta)=((\Th\circ\eta)^*(-\om),(\Th\circ\eta)^* (\Im\Om))=(-\om|_{\tilde{L}},\Im\Om|_{\tilde{L}}),$$ which proves $\tilde{L}$ is a special Lagrangian $3$-submanifold of $X$.

We now appeal to Definition \ref{acdfn4} to prove that $\tilde{L}$ is asymptotically cylindrical with decay rate $\ga$. $\Th\circ\eta:L\to\tilde{L}$ is a diffeomorphism, so let $\tilde{K'}=(\Th\circ\eta)(K')$; then $\tilde{K'}$ is a compact subset of $\tilde{L}$. Let $\tilde{\Phi}=\Th\circ\eta\circ\Phi:N\t\{p\}\t(R',\infty)\to \tilde{L}\setminus\tilde{K'}$; then $\tilde{\Phi}$ is a diffeomorphism. Finally, notice $\eta\in T^*L\cong\nu_L$, so $\xi^*\circ\Phi^*(\eta)$ is a section of $\nu_N\t\{p\}\t(R',\infty)$, $\xi$ defined by (\ref{xi}); $v$ is also a section of $\nu_N\t\{p\}\t(R',\infty)$, so we can define $\tilde{v}=v+\xi^*\circ\Phi^*(\eta)\in\nu_N\t\{p\}\t(R',\infty)$. With this data, $\tilde{L},\tilde{K'},\tilde{\Phi},\tilde{v}$, Diagram (\ref{aceq1}) commutes, that is, $\Psi\circ\exp_{\tilde{v}}=\tilde{\Phi}$ on $N\t\{p\}\t(R',\infty).$

We need finally to show that $\tilde{L}$ has the correct decay rate. First, $|\na^kv|=O(e^{\ga t})$ for all $k\geq0$ on $N\t\{p\}\t(R',\infty)$ since $L$ is asymptotically cylindrical with decay rate $\be<\ga$. Next, by Lemma \ref{smooth}, $\eta\in L^p_{l+2,\ga}(T^*L)$ for all $l\geq 1$; by the Sobolev Embedding Theorem \ref{embedding}, $|\na^k\eta|=O(e^{\ga\rho})$ (see Definition \ref{wss}) for all $k\geq0$ on $L$, so $|\na^k(\xi^*\circ\Phi^*)(\eta)|=O(e^{\ga t})$ for all $k\geq 0$ on $N\t\{p\}\t(R',\infty)$ since $\xi$ and $\Phi$ are asymptotically cylindrical. Thus, $|\na^k\tilde{v}|=O(e^{\ga t})$ on $N\t\{p\}(R',\infty)$ for all $k\geq0$ on $N\t\{p\}\t(R',\infty)$, and so $\tilde{L}$ is asymptotically cylindrical special Lagrangian with decay rate $\ga$; hence, $S:F^{-1}(0,0)\to M^{\ga}_L$ is well defined.

Conversely, assume that $\tilde{L}$ is close to $L$ in $M^{\ga}_L$, and let $\tilde{K},\tilde{\Phi},\tilde{v}$ be as in Definition \ref{acdfn4} for $\tilde{L}$. Then there exists a unique smooth section $\eta$ of the bundle $B_{\ep'}(T^*L)$ with $\Th\circ\eta:L\to\tilde{L}$ a diffeomorphism since $\tilde{L}$ and $L$ are $C^1$ close; moreover, because $\om|_{\tilde{L}}\equiv0$ and $\Im\Om_{\tilde{L}}\equiv0$,
$F(\eta)=(0,0)$.

Now here again, as in the above construction, we have $\tilde{v}=v+\xi^*\circ\Phi^*(\eta)$, and since $|\na^k\tilde{v}|=O(e^{\ga t})$ and $|\na^k v|=O(e^{\ga t})$ for all $k\geq 0$, we have $|\na^k\eta|=O(e^{\ga\rho})$ for all $k\geq 0$ on $L$. Unfortunately, this estimate is only good enough to show $\eta\in L^p_{2+l,\ga'}(B_{\ep'}(T^*L)$ for any $\ga'>\ga$; however, if $F'$ denotes $F$ with the new decay rate $\ga'$, we now have $\eta\in F'^{-1}(0,0)$. If we further take $\ga'>\ga$ with $[\ga,\ga']\cap\D_{(\d+*\d^*)_0}=\emptyset$, then all of the above arguments apply also to $F'$, so that both $F,F'$ are smooth, finite-dimensional and locally isomorphic to $\ker((\d+*\d^*)^p_{2+l,\ga})$ and $\ker((\d+*\d^*)^p_{2+l,\ga'})$ respectively. Recall that these kernels depend only on the connected components of $\D_{(\d+*\d^*)_0}$ in which the decay rates lie, so in fact these kernels are equal. Since $F^{-1}(0,0)\subseteq F'^{-1}(0,0)$, we have that $F^{-1}(0,0)=F'^{-1}(0,0)$ near $0$ and, hence, $\eta\in F^{-1}(0,0)\subset L^p_{2+l,\ga}(B_{\ep'}(T^*L))$.

It is left to consider the topology. Since we have identified submanifolds of $X$ with sections of the cotangent bundle $T^*L$ of $L$, we have an induced topology on $M^{\ga}_L$ coming from some Banach norm on the sections $\eta$ of $T^*L$. Recall that $F^{-1}(0,0)$ with the topology defined by the $L^p_{2+l,\ga}$ Banach norm is locally homeomorphic to the finite-dimensional vector space $\ker((\d+*\d^*)^p_{2+l,\ga})$. This shows that all Banach norms on the sections $\eta$ of $T^*L$ will induce the same topology on $M^{\ga}_L$, so that $S$ is indeed a local homeomorphism.
\end{proof}

\end{document}